\documentclass[12pt]{article}
\usepackage{array,amsmath,amssymb,amsthm,fullpage}

\begin{document}

\newtheorem{theorem}{Theorem}
\newtheorem{problem}{Problem}
\newtheorem{proposition}[theorem]{Proposition}
\newtheorem{question}{Question}
\theoremstyle{definition}
\newtheorem{remark}[theorem]{Remark}
\newtheorem{lemma}[theorem]{Lemma}
\def\st{\colon\,}
\def\PE#1#2#3{\prod_{#1=#2}^{#3}}
\def\ZZ{{\mathbb Z}}
\def\mod{\,{\rm mod}\,}
\def\C#1{\left|{#1}\right|}
\def\VEC#1#2#3{#1_{#2},\ldots,#1_{#3}}
\def\NN{{\mathbb N}}
\def\FR#1#2{\frac{#1}{#2}}
\def\FL#1{\lfloor{#1}\rfloor}
\def\e{{\rm e}}

\title{Cycles in Color-Critical Graphs}

\author{ 
Benjamin R. Moore\thanks{Computer Science Institute, Charles University,
Prague, Czech Republic: brmoore@iuuk.mff.cuni.cz.
Research supported by the Natural Sciences and Engineering Research Council of Canada.}\,,
Douglas B. West\thanks{Departments of Mathematics,
Zhejiang Normal University, Jinhua, China, and University of Illinois,
Urbana, IL: west@math.uiuc.edu.
Research supported by National Natural Science
Foundation of China grants NNSFC 11871439 and 11971439.
}}
\date{Revised November, 2021}
\maketitle

\baselineskip 16pt

\vspace{-2pc}
\begin{abstract}
Tuza [1992] proved that a graph with no cycles of length congruent to $1$
modulo $k$ is $k$-colorable.  We prove that if a graph $G$ has an edge $e$ such
that $G-e$ is $k$-colorable and $G$ is not, then for $2\le r\le k$, the edge
$e$ lies in at least $\PE i1{r-1}(k-i)$ cycles of length $1\mod r$ in $G$, and
$G-e$ contains at least $\frac12{\PE i1{r-1}(k-i)}$ cycles of length $0 \mod r$.

A $(k,d)$-coloring of $G$ is a homomorphism from $G$ to the graph $K_{k:d}$
with vertex set $\ZZ_{k}$ defined by making $i$ and $j$ adjacent if 
$d\le j-i \le k-d$.  When $k$ and $d$ are relatively prime, define $s$ by
$sd\equiv 1\mod k$.  A result of Zhu [2002] implies that $G$ is
$(k,d)$-colorable when $G$ has no cycle $C$ with length congruent to $is$
modulo $k$ for any $i\in \{1,\ldots,2d-1\}$.  In fact, only $d$ classes need be
excluded: we prove that if $G-e$ is $(k,d)$-colorable and $G$ is not, then $e$
lies in at least one cycle with length congruent to $is\mod k$ for some $i$ in
$\{1,\ldots,d\}$.  Furthermore, if this does not occur with
$i\in\{1,\ldots,d-1\}$, then $e$ lies in at least two cycles with length
$1\mod k$ and $G-e$ contains a cycle of length $0 \mod k$.
\end{abstract}

\section{Introduction}

One of the most fundamental results in graph theory is that graphs without odd
cycles are $2$-colorable.  This has been generalized in many ways.
Stong~\cite{Sto} proved that a graph is $k$-colorable if every vertex lies
in fewer than $\binom k2$ odd cycles.  Erd\H{o}s and Hajnal~\cite{EH} proved
that a graph having no odd cycle of length more than $2j-1$ is $2j$-colorable.
Tuza~\cite{Tuz} proved that a graph having no cycle with length congruent to
$1$ modulo $k$ is $k$-colorable.  Tuza also strengthened Minty's
Theorem~\cite{Min} that a graph $G$ is $k$-colorable if it has an orientation
in which no cycle of $G$ has more than $k-1$ times as many forward edges as
backward edges; Tuza showed that only cycles of length $1\mod k$ need be
considered in that computation.

More recent work has provided guarantees not only for existence of cycles
with certain lengths, but also lower bounds on the number of such cycles.
A graph is {\it $(k+1)$-critical} if it is not $k$-colorable but every proper
subgraph is $k$-colorable.  It is {\it doubly critical} if deleting the
endpoints of any edge reduces the chromatic number by $2$.  In studying a
weaker form of the Erd\H{o}s--Lov\'asz Tihany Conjecture that all doubly
critical graphs are complete, Kawarabayashi, Pedersen, and Toft~\cite{KPT}
showed that in a doubly critical graph with chromatic number $k+1$, every edge
lies in $\PE i1{l-2}(k-i)$ cycles of length $l$.  They used the technique of 
{\it generalized Kempe chains}, introduced as early as Neumann-Lara~\cite{Neu}
and named in Toft~\cite{Toft}.  Such a chain is a path following a particular
list of colors.

Without the doubly-critical requirement, in this paper we obtain cycles in
congruence classes rather than with specified lengths.  We illustrate the
generalized Kempe chain technique by first strengthening Tuza's basic result,
showing that for $2\le r\le k$ every edge in a $(k+1)$-critical graph lies in
at least $\PE i1{r-1}(k-i)$ cycles of length congruent to $1$ modulo $r$.
The statement is more general.
\begin{theorem}
For $2\le r\le k$ and $e\in E(G)$, if $G-e$ is $k$-colorable and $G$ is not,
then $e$ lies in at least $\PE i1{r-1}(k-i)$ cycles of length congruent to $1$
modulo $r$.  
\end{theorem}
Thus every graph with fewer than $(k-1)!$ cycles of length congruent to $1$
modulo $k$ is $k$-colorable.

\bigskip
We also apply our technique to the more general problem of $H$-coloring.
A {\it homomorphism} from a graph $G$ to a graph $H$ is a map
$\phi\st V(G)\to V(H)$ such that $uv\in E(G)$ implies $\phi(u)\phi(v)\in E(H)$.
A homomorphism into $H$ is also called an {\it $H$-coloring}.
A proper $k$-coloring is simply a $K_k$-coloring, where $K_k$ is the complete
graph with $k$ vertices.
The $H$-coloring problem becomes more complicated when $H$ is not complete
because there are more ways for a coloring to violate an edge.  An ordinary
proper coloring requires colors on adjacent vertices to be distinct, but this
is no longer enough.

The {\it circular clique} $K_{k:d}$ has vertex set $\ZZ_k$ and edge set
$\{ij\st d\le j-i\le k-d\}$.  The complete graph $K_k$ is simply the circular
clique $K_{k:1}$.  Homomorphism into $K_{k:d}$ is called {\it $(k,d)$-coloring},
and a graph having a $(k,d)$-coloring is {\it $(k,d)$-colorable}.  When $G$ has
an edge, $(k,d)$-coloring requires $k\ge2d$.
The {\it circular chromatic number} $\chi_c$ of a graph $G$ is the least $k/d$
such that $G$ is $(k,d)$-colorable.  In particular, if $k'/d'\le k/d$, then
$K_{k':d'}$ is $K_{k:d}$-colorable.  See Zhu~\cite{Zhu1,Zhu2} for surveys on
this topic.

Zhu~\cite{Zhu3} extended Tuza's result to circular coloring.  Given an 
orientation of a graph $G$, and given a cycle $C$ in $G$ viewed in a consistent
direction, let $C^-$ denote the set of edges in $C$ oriented oppositely to
their orientation of $G$.  Zhu proved that if $G$ has an orientation such that
$\C{E(C)}/|C^-|\le k/d$ for every cycle $C$ (in each direction) such that
$d\C{E(C)}$ is congruent modulo $k$ to some value in $\{1,\ldots,2d-1\}$, then
$G$ is $(k,d)$-colorable.

When $d$ and $k$ are relatively prime, let $s$ be the congruence class such
that $sd\equiv 1\mod k$.  It follows from Zhu's result that if $G$ has no cycle
with length congruent to $is$ modulo $k$ for any $i$ with $1\le i\le 2d-1$,
then $G$ is $(k,d)$-colorable.  That is, $(k,d)$-colorability holds when $2d-1$
congruence classes of cycle lengths modulo $k$ are forbidden.  Our result
implies that it suffices to exclude $i\in\{1,\ldots,d\}$.  In particular, if
$G-e$ is $K_{k:d}$-colorable and $G$ is not, then $e$ lies in a cycle with
length congruent to $is\mod k$ for some $i$ in $\{1,\ldots,d\}$.  Note that
$ds\equiv 1\mod k$.  We further show that if $G$ has no cycle through $e$ with
length $is\mod k$ when $i\in\{1,\ldots,d-1\}$, then $e$ lies in at least two
cycles with length $1\mod k$ and $G-e$ contains a cycle of length $0 \mod k$.
(Note that $(k,d)$-coloring forbids loops.)

The special case of $K_{(2d+1):d}$-coloring shows that our result is sharp.
Here $-2d\equiv 1\mod(2d+1)$, so $s=-2$, and we seek a cycle length congruent
to $-2i$ for some $i\in\{1,\ldots,k\}$.  These lengths are the odd values from
$2d-1$ to $1$.  The graph $K_{(2d+1):d}$ is isomorphic to the odd cycle
$C_{2d+1}$.  Thus $\chi_c(G)\le 2+1/d$ for a $C_{2d+1}$-colorable graph
$G$.  All shorter odd cycles are critical non-$C_{2d+1}$-colorable graphs (we
can also view a single vertex with a loop as a degenerate such example).  This
means that although possibly only one cycle length among the listed classes of
lengths occurs, we cannot omit any of those classes from the list.

\bigskip
Returning to the topic of proper $k$-coloring, there are further questions
to ask.  An edge in the complete graph $K_{k+1}$ lies in exactly $(k-1)!$
cycles of length $1\mod k$.  When $r$ is at most $k/2$, the guarantee
of $\PE i1{r-1}(k-i)$ cycles of length $1\mod r$ is not sharp in $K_{k+1}$,
because cycles of length $r+1$ and cycles of length $2r+1$ both count.

\begin{question}
Can the guarantee of $\PE i1{r-1}(k-i)$ cycles of length $1\mod r$ through
each edge be sharp for non-complete $(k+1)$-critical graphs
when $k/2<r<k$?
\end{question}

We can also consider other congruence classes.  In fact, $K_{k+1}$ has cycles
in all congruence classes modulo $r$ for $2\le r\le k$ except for the class of
$2$ modulo $k$.  In a paper on cycle lengths in directed graphs, Chen, Ma, and
Zang \cite{Jiema1} proved that for integers $l$ and $k$ with $1\le l\le k$ and
$k\ge2$, a graph containing no cycle of length congruent to $l$ modulo $k$ is
$k$-colorable if $l\ne2$, and is $(k+1)$-colorable if $l=2$.  This proved a
strong form of a conjecture of Diwan, Kenkre, and Vishwanathan~\cite{DKV} and
answered a question asked by Tuza~\cite{Tuz}.

A recent paper of Gao, Huo, Lui, and Ma~\cite{Jiema2} resolves almost all the
existence questions in a very strong way.  Stated in our language, they proved
that a non-$k$-colorable graph contains $k-1$ cycles of consecutive lengths.
This covers all but one congruence class modulo $k$ and all congruence classes
for smaller moduli.  Non-$k$-colorable graphs have $(k+1)$-critical subgraphs,
which have minimum degree at least $k$; they proved also that $3$-connected
nonbipartite graphs with minimum degree at least $k$ have cycles with $k-1$
consecutive lengths.  Also, minimum degree at least $k$ guarantees cycles of
all even lengths modulo $k-1$, extending to all lengths modulo $k-1$ in the
$2$-connected nonbipartite case.  Finally, for $k\ge3$ every $k$-connected
graph has a cycle whose length is a multiple of $k$.  These results were
variously conjectured by Sudakov and Verstra\"ete, by Bondy and Vince, by
Thomassen, and by Dean.  In~\cite{Jiema2} they are proved by a unified
approach, but the argument is quite long.

A subsequent paper by Gao, Huo, and Ma~\cite{Jiema3} resolved the remaining
question about $2\mod k$, proving that when $k\ge6$ every non-$k$-colorable
graph not having $K_{k+1}$ as a block contains $k$ cycles of consecutive
lengths.  Hence it has cycles in all congruence classes with modulus at most
$k$.  The case $2\mod k$ when $k=3$ follows from a combination of results of
Saito~\cite{Sai} and Dean, Kaneko, Ota, and Toft~\cite{DKOT}.  For
$k\in\{4,5\}$, the authors of \cite{Jiema3} state that their method works but
yields a proof that was too long to include.  Their result culminates a long
series of conjectures and theorems on cycle lengths in color-critical graphs by
many researchers; see~\cite{Jiema3} for the history and further references.
%
%
%

One may also wonder whether the cycle length guarantees follow from weaker
hypotheses.  Since $(k+1)$-critical graphs are $k$-edge-connected, one may
wonder whether being $k$-chromatic and critically $k$-edge-connected is enough.
Already this fails when $r=k$ and we ask for just one cycle.  The Petersen
graph is $3$-colorable (but not $3$-critical) and is critically
$3$-edge-connected.  However, it has no $4$-cycle, no $7$-cycle, and no
$10$-cycle, so it has no cycle of length congruent to $1$ modulo $3$.

\section{Proper Coloring}

We consider the cycles forced when deletion of an edge reduces the chromatic
number.  

\begin{theorem}\label{kcol}
Fix $r,k\in\NN$ with $2\le r\le k$, and let $e$ be an edge in a graph $G$.  If
$G-e$ is $k$-colorable and $G$ is not, then $e$ belongs to at least
$\PE i1{r-1}(k-i)$ cycles in $G$ having lengths congruent to $1$ modulo $r$.
\end{theorem}

\begin{proof}
Let $[k]=\{1,\ldots,k\}$.  Fix a proper $k$-coloring $\phi$ of $G-e$ with
colors in $[k]$.  Let $x$ and $y$ be the endpoints of $e$.  We obtain a cycle
through $e$ for each cyclic list of $r$ members of $[k]$ containing $\phi(x)$.
Starting with $\phi(x)$, the cyclic list $\sigma$ can be formed in $\PE
i1{r-1}(k-i)$ ways.

Given such $\sigma$, define the {\it $\sigma$-subdigraph} of $G$ generated by
$\phi$ to be the digraph $D_\sigma$ with vertex set $V(G)$ such that $uv$ is an
edge in $D_\sigma$ if and only if $uv\in E(G)$ and $\sigma(\phi(u))=\phi(v)$.
Let $F$ be the subdigraph of $D_\sigma$ induced by all vertices reachable from
$x$ by paths in $D_\sigma$.

Define a recoloring $\phi'$ of $G-e$ by $\phi'(u)=\sigma(\phi(u))$ for
$u\in V(F)$ and $\phi'(u)=\phi(u)$ for $u\in V(G)-V(F)$.  An edge is improperly
colored by $\phi'$ only if the color of one endpoint remains fixed and the 
other changes into it, but then the oriented version of the edge lies in $F$
and both endpoints change color.  Thus $\phi'$ is a proper $k$-coloring of
$G-e$.

Since $G$ is not $k$-colorable, $\phi(x)=\phi(y)$.  Also $\phi'(x)\ne \phi(x)$.
If $y\notin V(F)$, then we have $\phi'(x)\ne\phi(x)=\phi(y)=\phi'(y)$, and
$\phi'$ is a proper coloring of $G$, which by hypothesis does not exist.  Hence
$y\in V(F)$, meaning that $y$ is reachable from $x$ via a path in $F$.

Since paths in $F$ follow colors according to $\sigma$, and $\phi(y)=\phi(x)$,
the length of any $x,y$-path in $F$ is a multiple of $r$, and the cycle in $G$
completed by adding the edge $yx$ has length congruent to $1$ modulo $r$.
Furthermore, since the coloring $\phi$ is fixed, the resulting $x,y$-paths in
$G$ are distinct for distinct choices of $\sigma$.
Hence we obtain $\PE i1{r-1}(k-i)$ cycles through $e$.
\end{proof}

Our result was motivated by a similar quantitative argument by Brewster,
McGuinness, Moore, and Noel~\cite{BMMN}.  We state it in our terminology to
generalize it.  For $k>2$, they showed that if $G$ is not $k$-colorable but
$G-xy$ is $k$-colorable, then $G-xy$ contains at least $(k-1)!/2$ cycles with
lengths divisible by $k$.  The early paper of Tuza~\cite{Tuz} notes that Toft
and Tuza had observed for $k>2$ that every non-$k$-colorable graph contains a
cycle whose length is divisible by $k$.

\begin{theorem}
For $3\le r\le k$, if a graph $G$ is not $k$-colorable but $G-e$ is
$k$-colorable, where $e\in E(G)$, then $G-e$ contains at least
$\FR12\PE i1{r-1}(k-i)$ cycles whose lengths are divisible by $r$, none of
which contain $e$.
\end{theorem}
\begin{proof}
Let $x$ and $y$ be the endpoints of $e$.  Again fix a proper $k$-coloring
$\phi$ of $G-e$ and a cyclic permutation $\sigma$ of a set of $r$ colors
containing $\phi(x)$.  Define the digraph $D_\sigma$ as above.  Note again that
$\phi(x)=\phi(y)$, since $G$ is not $k$-colorable.

If $D_\sigma$ is acyclic, then we recolor $G$ by again changing the color on
$v$ from $\phi(v)$ to $\sigma(\phi(v))$, but this time we perform the change
{\it one vertex at a time}, always changing the color at a sink of the
unchanged subgraph.  At each step we have a proper $k$-coloring of $G-e$.  If
at some point the color on $x$ or $y$ changes, then we have produced a proper
$k$-coloring of $G$.  Since $G$ has no such coloring, $D_\sigma$ must contain
a cycle.  Since $\phi(x)=\phi(y)$, the edge $e$ does not appear in $D_\sigma$,
so such cycles do not contain $e$.

The length of any cycle in $D_\sigma$ is a multiple of $r$.  However, a
cyclic permutation and its reverse will select the same cycle in $G$, because
the corresponding digraphs are obtained from each other by reversing all the
edges.  Hence we are in fact guaranteed $\FR12\PE i1{r-1}(k-i)$ cycles whose lengths are
multiples of $r$, and none of these cycles contain $e$.
\end{proof}

We have guaranteed $(k-1)!$ cycles of length $1\mod k$ in a $(k+1)$-critical
graph.  We next present a probabilistic argument that guarantees $k!/2$,
suggested by a referee.  Although $k!/2>(k-1)!$ when $k\ge3$, this argument
does not yield $k!/2$ such cycles {\it through every edge} in a
$(k+1)$-critical graph, so neither result implies the other.

\begin{theorem}
For $k\ge3$, a non-$k$-colorable graph has at least $k!/2$ cycles with lengths
congruent to $1$ modulo $k$, with equality for $k\ge4$ only when these cycles
all have length $k+1$.
\end{theorem}
\begin{proof}
Randomly order the vertices and orient each edge toward its later endpoint in
the order.  We will bound the probability that a given cycle of length
$1\mod k$ (or its reverse) has more than $k-1$ times as many forward edges as
backward edges.  If this probability is at most $2/k!$ and there are fewer than
$k!/2$ such cycles, then some orientation has no cycle of length $1\mod k$ with
more than $k-1$ times as many forward edges as backward edges.  By Tuza's
strengthening of Minty's Theorem, a graph with such an orientation is
$k$-colorable.

Let $C$ be a cycle of length $qk+1$.  Having more than $k-1$ times as many
forward edges as backward edges means that if following the vertices along $C$
involves at most $q$ backward steps and more than $(k-1)q$ forward steps in the
ordering.  Since the vertices outside the cycle are irrelevant, it suffices to
show that at most $(qk+1)!/k!$ of the $(qk+1)!$ orderings of $\VEC v1{qk+1}$
have at most $q$ instances of $v_i$ preceding $v_{i-1}$.

From one backward step to the next is an increasing run.  Hence to form an
ordering with at most $q$ backward steps we assign the {\it positions} of
vertices to bins $1$ through $q$ and place the positions within a bin in
increasing order.  This produces a list $\VEC\sigma1{qk+1}$ of the
positions $1$ through $qk+1$.

Now form the vertex permutation by putting vertex $v_i$ in position $\sigma_i$.
Vertices whose positions are in a single bin form a forward path in the
orientation, and backward steps only occur when starting a new bin.
Furthermore, every ordering of the vertices for which the resulting orientation
has at most $q$ backward steps arises in this way.  When the least position in
the next bin is higher than the last position in the current bin, there are
fewer backward steps, so we have included the vertex orderings where the cycle
has fewer backward steps.

Since we can start indexing the given cycle at any of its $qk+1$ vertices, a
given distribution of positions to $q$ bins produces bad orientations for the
cycle in $qk+1$ ways.  The probability that this cycle has too few back edges
in the random vertex ordering is thus bounded by $q^{qk+1}/(qk)!$.
We multiply by $2$ since the same cycle also arises in the opposite direction.

It thus suffices to show
\begin{equation}\label{ineq}
\frac{q^{qk+1}}{(qk)!} \le \frac1{k!}.
\end{equation}
Equality holds when $q=1$.  Strict inequality for other cases yields the
additional observation that if a non-$k$-colorable graph has only $k!/2$ cycles
of length congruent to $1$ modulo $k$, then those cycles all must have length
exactly $k+1$.  There are $k!/2$ cycles of length $k+1$ in $K_{k+1}$, so the
result is sharp.

The inequality \eqref{ineq} fails when $(q,k)=(2,3)$, where the value of
$q^{qk+1}/(qk)!$ is $8/45$, which exceeds $1/6$, but this is small enough.
The probability that a given $7$-cycle followed in order has too few back edges
is bounded by $8/45$, and for cycles of other lengths congruent to $1\mod 3$
the probability of having too few back edges will be bounded by $1/6$.  Since
we can follow a cycle in either direction, we change these bounds to $16/45$
and $1/3$.  Therefore, if a graph has only two cycles of length $1\mod 3$, the
expected number of bad cycles is bounded by $32/45$, so some vertex ordering
guarantees 3-colorability. 

It thus suffices to have \eqref{ineq} when $k\ge4$ and $q\ge2$ (except
$(q,k)=(2,3)$).  We give an approximate computation that is easy to make precise.
Stirling's Approximation is
$$
n!=\left(\frac{n}{\e}\right)^n\sqrt{2\pi n}\left[1+\frac1{12n}
+\frac1{288n^2}-\frac{139}{51840n^3}-O(n^{-4})\right].
$$
We keep only the first term, which provides a lower bound on $n!$

Rewriting the desired inequality, we seek
$q^{qk+1}\le (qk)!/k!$.  By Stirling's Approximation,
$$
\frac{(qk)!}{k!}
\approx \frac{(qk/\e)^{qk}}{(k/\e)^k}\frac{\sqrt{2\pi qk}}{\sqrt{2\pi k}}
=q^{qk+1/2}\left(\frac k\e\right)^{(q-1)k}.
$$
Thus the inequality we need is (roughly) $\sqrt q< (k/\e)^{(q-1)k}$.  The
right side increases rapidly with $k$.  When $k=3$, it equals $1.104^{3(q-1)}$.
Already when $q=3$ this is greater than $\sqrt q$, and similarly the inequality
holds for $(q,k)=(2,4)$.  To make the approximate argument precise, note that
in the numerator our approximation to $(qk)!$ is already less than $(qk)!$.  We
have the slack to use something slightly larger than $k!$ in the denominator,
which yields $q^{qk+1}<x<(qk)!/k!$ for some $x$.
\end{proof}

We believe that the characterization of equality also holds when $k=3$,
but there the argument above only restricts to lengths $4$ and $7$.

\section{Circular Coloring}

In this section we consider the analogous problem for $(k,d)$-coloring.
We will only use color cycles of the form $(0,d,2d,\ldots,-d)$ and their
reverse, so we get existence results rather than quantitative results.
Nevertheless, they are sharp in terms of the number of classes allowed,
as discussed in the introduction.

The proof may require many steps of recoloring to find a desired cycle.  This
is inherently necessary, because a $C_{2d+1}$-coloring of $C_{2d-1}-e$ may
alternate $0$ and $d$ along the path.

\begin{theorem}\label{circular}
Given $k$ and $d$ relatively prime with $k>2d$, let $s$ be the element of
$\ZZ_k$ such that $sd\equiv 1\mod k$.  For an edge $e$ in a graph $G$, if
$G-e$ is $K_{k:d}$-colorable and $G$ is not, then $e$ lies in a cycle in $G$ of
length congruent to $is\mod k$ for some $i$ in $\{1,\ldots,d\}$.
\end{theorem}

\begin{proof}
Fix a $K_{k:d}$-coloring $\phi$ of $G-e$.  Let $x$ and $y$ be the endpoints
of $e$.  By cyclic symmetry, we may assume $\phi(y)=0$.  Since $G$ is not
$K_{k:d}$-colorable, $\phi(x)\in\{0,\pm1,\ldots,\pm(d-1)\}$.
Let $\sigma$ be the cyclic permutation $(0,d,2d,\ldots,-d)$ of colors.
Define the digraph $D_\sigma$ as in Theorem~\ref{kcol}, and let $F$ be
the subdigraph of $D_\sigma$ induced by all vertices reachable from $x$ in
$D_\sigma$.

Given $\phi$, define $\phi'$ on $G-e$ by letting $\phi'(v)=\phi(v)+1$ for
$v\in V(F)$ and $\phi'(v)=\phi(v)$ for $v\notin V(F)$.  We claim that $\phi'$
is a $K_{k:d}$-coloring of $G-e$.  First, edges within $F$ or in $G-V(F)$
remain properly colored.  When $v\in V(F)$, the exploration of $D_\sigma$
extends along the edge $vw$ if $\phi(w)-\phi(v)=d$.  Since
$\phi(w)-\phi(v)\in\{d,d+1,\ldots,k-d\}$ for $vw\in E(G-e)$, having $v\in V(F)$
and $w\notin V(F)$ requires $\phi(w)-\phi(v)\in\{d+1,\ldots,k-d\}$.  Now
$\phi'(w)-\phi'(v)\in\{d,\ldots,k-d-1\}$, so such edges are also properly
colored in $\phi'$.

We will consider cases where $\phi(x)=j$, for $0\le j\le d-1$.
For the case $\phi(x)=-j$ with $1\le j\le d-1$, add $j$ to the color at each
vertex to obtain $\phi(x)=0$ and $\phi(y)=j$, and then interchange the roles of
$x$ and $y$ and apply the argument below.

When $\phi(x)=j$ and $\phi(y)=0$, we claim that $G$ has a cycle through $xy$
with length congruent to $is$ modulo $k$ for some $i$ in $\{1,\ldots,d-j\}$.
Note first that if $F$ has an $x,y$-path of length $r$, then 
$rd\equiv -j\mod k$, since each edge increases the color value by $d$.
Multiplying by $s$ yields $r\equiv -js\mod k$.  Since $sd\equiv1$,
adding $1$ to each side to compute the length of the cycle yields
$r+1\equiv -js+ds\equiv (d-j)s\mod k$.

We now prove the claim by induction on $d-j$.  
First consider $j=d-1$.  If $y\notin V(F)$, then $\phi'$ is a
$K_{k:d}$-coloring of $G$, since $\phi'(x)=d$ and $\phi'(y)=0$.  Hence 
$y\in V(F)$.  Now by the computation above we have a cycle through $yx$
with length congruent to $1s\mod k$.

Now suppose $j<d-1$.  If $y\in V(F)$, then the computation yields a cycle
through $yx$ with length congruent to $(d-j)s\mod k$.  Hence we may assume
$y\notin V(F)$.  Now $\phi'$ is a $K_{k:d}$-coloring of $G-xy$ with
$\phi'(x)=j+1$ and $\phi'(y)=0$.  The induction hypothesis, applied to
$\phi'$ with $\phi'(x)=j+1$, now implies that $G$ has a cycle through $xy$ with
length congruent to $is\mod k$ for some $i$ in $\{1,\ldots,d-j-1\}$.
Including $d-j$ in the set thus covers all cases to complete the induction step.
\end{proof}

Note that the proof of Theorem~\ref{circular} gives more detailed statements.
In particular, if the $K_{k:d}$-coloring of $G-e$ gives distinct colors to the
endpoints of $e$, then $G$ has a cycle through $e$ of length $is\mod{k}$ for
some $i$ in $\{1,\dots,d-1\}$.  The next result shows that if no such cycle
occurs, then we can find an extra cycle through $e$ of the remaining congruence
class, plus one avoiding $e$ with length divisible by $k$.

\begin{proposition}
If in the setting of Theorem~\ref{circular}, $e$ does not lie in a cycle
with length congruent to $is\mod k$ for some $i$ in $\{1,\ldots,d-1\}$, then
$e$ lies in at least two cycles of length $1\mod k$ and $G-e$ contains a
cycle of length $0 \mod k$.
\end{proposition}
\begin{proof}
With the possibilities $i\in\{1,\dots,d-1\}$ excluded in the argument of
Theorem~\ref{circular}, the remaining case is $j=0$ and $y\in V(F)$.
To reach $y$ from $x$ along steps of value $+d$, the number of steps must
be a multiple of $k$, since $\phi(x)=\phi(y)$, and adding $e$ completes a
cycle.  Under $\sigma^{-1}$, using the same $(k,d)$-coloring $\phi$ of $G-e$
and starting again from $x$ yields a second cycle of length $1\mod k$ through
$xy$.

Furthermore, if $G-e$ has no cycle of length $0\mod k$, then $D_\sigma$ is
acyclic.  Working backward from sinks, we can add $1$ to the color of each
reached vertex, \textit{one vertex at a time}, always maintaining a
$(k,d)$-coloring of $G-e$, until $x$ or $y$ changes color.  This reduces the
problem to the case $j>0$.  Since in this case $G$ has no cycle of length
$is$ with $i\in\{1,\ldots, d-1\}$, the previous arguments produce a
$(k,d)$-coloring of $G$, which by hypothesis does not exist.  Therefore, in
fact $G-e$ also contains a cycle of length $0 \mod k$. 
\end{proof}

In these arguments, we have not used cycles in $K_{k:d}$ other than that
generated by $d$ or $-d$.  When $k=2d+1$, these two are the only permutations
yielding cycles in the host graph, and that is why our sharpness examples in the
introduction are for $C_{2d+1}$-coloring.  In that case $s=-2$.  The set of
cycle lengths that cannot be avoided are the congruence classes $-2i\mod(2d+1)$
for $1\le i\le d$, and if there are no cycles through $e$ in the classes with
$1\le i\le d-1$, then we obtain two cycles with lengths $1\mod(2d+1)$.

Other cycles in the host graph can yield other sets of forced cycles, but
the key is designing a recoloring that produces another $H$-coloring of
$G-e$.

\end{document}